\documentclass[12pt]{article}
\usepackage{amsmath, amsthm}\usepackage{enumerate}\usepackage{amssymb}\usepackage{bbold}
\usepackage{bbm}\usepackage{dsfont}\usepackage{color}
\newtheorem{theorem}{Theorem}[section]
\newtheorem{corollary}[theorem]{Corollary}
\newtheorem{lemma}[theorem]{Lemma}

\newtheorem{proposition}[theorem]{Proposition}

\theoremstyle{definition}

\newtheorem{definition}[theorem]{Definition}
\newtheorem{example}[theorem]{Example}

\newtheorem{remark}[theorem]{Remark}

\DeclareMathOperator{\co}{\xrightarrow[]{o}}

\DeclareMathOperator{\stmu}{\downarrow^{st_\mu}}
\DeclareMathOperator{\stmuc}{\xrightarrow[]{st_\mu}}

\begin{document}

\title{Statistical convergence of nets in Riesz spaces}\maketitle\author{\centering{{Abdullah Ayd\i n$^{1,*}$, Fatih Temizsu$^2$\vspace{2mm} \\
\small $^1$Department of Mathematics, Mu\c{s} Alparslan University, Mu\c{s}, Turkey \\  \small a.aydin@alparslan.edu.tr\\
\small $^2$Department of Mathematics, Bingöl University, Bingöl, Turkey \\  \small ftemizsu@bingol.edu.tr \\$^*$Corresponding Author}
		
\abstract{The statistical convergence is defined for sequences with the asymptotic density on the natural numbers, in general. In this paper, we introduce the statistical convergence for nets in Riesz spaces by using the finite additive measures on directed sets. Moreover, we give some relations among the statistical convergence and the lattice properties such as the order convergence and lattice operators.}\\
\vspace{2mm}

\vspace{2mm}

\section{Introduction}
The statistical convergence of sequences is handled together with the asymptotic (or, natural) density of sets on the natural numbers $\mathbb{N}$. On the other hand, Connor introduced the notion of statistical convergence of sequences with finitely additive set function \cite{Con,Con2}. After then, some similar works have been done (cf. \cite{CLL,DO,Mor}). Also, several applications and generalizations of the statistical convergence of sequences have been investigated by several authors (cf. \cite{Aydn1,Aydn2,AC,Fast,M,Mor,St,SP}). However, as far as we know, the concept of statistical convergence related to nets has not been done except for the paper \cite{MDM}, in which the asymptotic density of a directed set $(D,\leq)$ was introduced by putting a special and strong rule on the directed sets such as the set $\{\alpha\in D:\alpha\leq\beta\}$ is finite and the set $\{\alpha\in D:\alpha\geq\beta\}$ is infinite for each element $\beta$ in a directed set $(D,\leq)$. By the way, we aim to introduce a general concept of statistical convergent nets thanks to a new notion called directed set measure.

Recall that a binary relation \textquotedblleft$\leq$\textquotedblright \ on a set $A$ is called a {\em preorder} if it is reflexive and transitive. A non-empty set $A$ with a preorder binary relation "$\leq$" is said to be a \textit{directed upwards} (or, for short, \textit{directed set}) if for each pair $x,y\in A$ there must exist $z\in A$ such that $x\leq z$ and $y\leq z$. Unless otherwise stated, we consider all directed sets as infinite. For given elements $a$ and $b$ in a preorder set $A$ such that $a\leq b$, the set $\{x\in A:a\leq x\leq b\}$ is called an \textit{order interval} in $A$. A subset $I$ of $A$ is called an \textit{order bounded set} whenever $I$ is contained in an order interval. 

A function whose domain is a directed set is said to be a \textit{net}. A net is briefly abbreviated as $(x_\alpha)_{\alpha\in A}$ with its directed domain set $A$. $(A,\leq_A)$ and $(B,\leq_B)$ be directed sets. Then a net $(y_\beta)_{\beta\in B}$ is said to be a \textit{subnet} of a net $(x_\alpha)_{\alpha\in A}$ in a non empty set $X$ if there exists a function $t:B\to A$ such that $y_\beta=x_{t(\beta)}$ for all $\beta\in B$, and also, for each $\alpha\in A$ there exists $\beta_\alpha\in B$ such that $\alpha\leq t(\beta)$ for all $\beta\geq \beta_\alpha$ (cf. \cite[Def.3.3.14]{Rud}). It can be seen that $\{t(\beta)\in A:\beta_\alpha\leq \beta\}\subseteq\{\hat{\alpha}\in A:\alpha\leq\hat{\alpha}\}$ holds for subnets. 

A real vector space $E$ with an order relation \textquotedblleft$\leq$\textquotedblright \ is called an {\em ordered vector space} if, for each $x,y\in E$ with $x\leq y$, $x+z\leq y+z$ and $\alpha x\leq\alpha y$ hold for all $z\in E$ and $\alpha \in \mathbb{R}_+$. An ordered vector space $E$ is called a {\em Riesz space} or {\em a vector lattice} if, for any two vectors $x,y\in E$, the infimum and the supremum
$$
x\wedge y=\inf\{x,y\} \ \ \text{and} \ \ x\vee y=\sup\{x,y\}
$$
exist in $E$, respectively.	A Riesz space is called a \textit{Dedekind complete} if every nonempty bounded from the above  set has a supremum (or, equivalently, whenever every nonempty bounded below subset has an infimum). A subset $I$ of a Riesz space $E$ is said to be a \textit{solid} if, for each $x\in E$ and $y\in I$ with $|x|\leq|y|$, it follows that $x\in I$. A solid vector subspace is called an \textit{order ideal}. A Riesz space $E$ has the \textit{Archimedean} property provided that $\frac{1}{n}x\downarrow0$ holds in $E$ for each $x\in E_+$. In this paper, unless otherwise stated, all Riesz spaces are assumed to be real and Archimedean. We remind the crucial notion of Riesz spaces (cf. \cite{AB,ABPO,LZ,Za}).

\begin{definition}\label{order conv}
	A net $(x_\alpha)_{\alpha\in A}$ in a Riesz space $E$ is called {\em order convergent} to $x\in E$ if there exists another net $(y_\alpha)_{\alpha\in A}\downarrow 0$ (i.e., $\inf y_\alpha=0$ and $y_\alpha\downarrow$) such that $|x_\alpha-x|\le y_\alpha$ holds for all $\alpha\in A$.
\end{definition}

We refer to the reader for some different types of the order convergence and some relations among them to \cite{AS}.


\section{The $\mu$-statistically order convergence}\label{sec2}
We remind that a map from a field $\mathcal{M}$ to $[0,\infty]$ is called {\em finitely additive measure} whenever $\mu(\emptyset)=0$ and $\mu(\cup_{i=1}^{n}E_i)=\sum_{i=1}^{n}\mu(E_i)$ for all finite disjoint sets $\{E_i\}_{i=1}^n$ in $\mathcal{M}$ (cf. \cite[p.25]{Folland}). Now, we introduce the notion of measure on directed sets.
\begin{definition}\label{directed set measure}
	Let $A$ be a directed set and $\mathcal{M}$ be a subfield of $\mathcal{P}(A)$. Then
	\begin{enumerate}
		\item[(1)] \ an order interval $[a,b]$ of $A$ is said to be a {\em finite order interval} if it is a finite subset of $A$;
		\item[(2)] \ $\mathcal{M}$ is called an {\em interval field} on $A$ whenever it includes all finite order intervals of $A$;
		\item[(3)] \ a finitely additive measure $\mu:\mathcal{M}\to[0,1]$ is said to be a {\em directed set measure} if $\mathcal{M}$ is an interval field and $\mu$ satisfies the following facts: $\mu(I)=0$ for every finite order interval $I\in \mathcal{M}$; $\mu(A)=1$; $\mu(C)=0$ whenever $C\subseteq B$ and $\mu(B)=0$ holds for $B,C\in \mathcal{M}$.
	\end{enumerate} 
\end{definition}

\begin{example}\label{d set measure exm}
	Consider the directed set $A:=\mathbb{N}$ and define a measure $\mu$ from $2^\mathbb{N}$ to $[0,1]$ denoted by $\mu(A)$ as the Banach limit of $\frac{1}{k}|A\cap\{1,2,\dots,k\}|$ for all $A\in 2^\mathbb{N}$. Then one can see that $\mu(I)=0$ for all finite order interval sets because of $\frac{1}{k}|I\cap\{1,2,\dots,k\}|\leq \frac{1}{k} I\to 0$. Also, it follows from the properties of Banach limit that $\mu(\mathbb{N})=1$ and $\mu(A\cup B)=\mu(A)+\mu(B)$ for disjoint sets $A$ and $B$. Thus, $\mu$ is finitely additive, and so, it is a directed set measure.
\end{example}

Throughout this paper, the vertical bar of a set will stand for the cardinality of the given set and $\mathcal{P}(A)$ is the power set of $A$. Let give an example of directed set measure for an arbitrary uncountable set.
\begin{example}\label{net give measure}
	Let $A$ be an uncountable directed set. Consider a field $\mathcal{M}$ consists of countable or co-countable (i.e., the complement of the set is countable) subsets of $A$. Then $\mathcal{M}$ is an interval field. Thus, a map $\mu$ from $\mathcal{M}$ to $[0,1]$ defined by $\mu(C):=0$ if $C$ is a countable set, otherwise $\mu(C)=1$. Hence, $\mu$ is a directed set measure. 
\end{example}

In this paper, unless otherwise stated, we consider all nets with a directed set measure on interval fields of the power set of  the index sets. Moreover, in order to simplify presentation, a directed set measure on an interval field $\mathcal{M}$ of the directed set $A$ will be expressed briefly as a measure on the directed set $A$. Motivated from \cite[p.302]{Fridy}, we give the following notion.

\begin{definition}\label{decreasin to}
	If a net $(x_\alpha)_{\alpha\in A}$ satisfies a property $P$ for all $\alpha$ except a subset of $A$ with measure zero then we say that $(x_\alpha)_{\alpha\in A}$ satisfies the property $P$ for almost all $\alpha$, and we abbreviate this by a.a.$\alpha$.
\end{definition}

Recall that the asymptotic density of a subset $K$ of natural numbers $\mathbb{N}$ is defined by
\[
\delta(K):=\lim_{n\rightarrow \infty}\frac{1}{n}\left \vert \left \{  k\leq n:k\in
A\right \}  \right \vert.
\]
We refer the reader for an exposition on the asymptotic density of sets in $\mathbb{N}$ to \cite{FS,SAL}. It can be observed that finite subsets and finite order intervals in natural numbers coincide. Thus, we give the following observation.
\begin{remark}\label{natural density properties}
	It is clear that the asymptotic density of subsets on $\mathbb{N}$ satisfies the conditions of a directed set measure when $\mathcal{P}(\mathbb{N})$ is considered as an interval field on the directed set $\mathbb{N}$. Thus, it can be seen that the directed set measure is an extension of the asymptotic density.
\end{remark}

Remind that a sequence $(x_n)$ in a Riesz space $E$ is called {\em statistically monotone decreasing} to $x\in E$ if there exists a subset $K$ of $\mathbb{N}$ such that $\delta(K)=1$ and $(x_{n_k})$ is decreasing to $x$, i.e., $(x_{n_k})_k\downarrow$ and $\inf(x_{n_k})=x$ (cf. \cite{Aydn2,SP}). Now, by using the notions of measure on directed sets and the statistical monotone decreasing which was introduced in \cite{SAL} for real sequences, we introduce the concept of the statistical convergence of nets in Riesz spaces.

\begin{definition}\label{conv def}
	Let $E$ be a Riesz space and $(p_\alpha)_{\alpha\in A}$ be a net in $E$ with a measure $\mu$ on the index set $A$. Then $(p_\alpha)_{\alpha\in A}$ is said to be {\em $\mu$-statistically decreasing} to $x\in E$ whenever there exists $\Delta\in \mathcal{M}$ such that $\mu(\Delta)=1$ and $(p_{\alpha_\delta})_{\delta\in \Delta}\downarrow x$. Then it is abbreviated as $(p_\alpha)_{\alpha\in A}\stmu x$.
\end{definition}

We denote the class of all $\mu$-statistically decreasing nets in a Riesz space $E$ by $E_{st_\mu\downarrow}$, and also, the set $E_{st_\mu\downarrow}\{0\}$ denotes the class of all $\mu$-statistically decreasing null nets on $E$. On the other hand, one can observe that $\mu(\Delta^c)=\mu(A-\Delta)=0$ whenever $\mu(\Delta)=1$ because of $\mu(A)=\mu(\Delta\cup \Delta^c)=\mu(\Delta)+\mu(\Delta^c)$. We consider Example \ref{net give measure} for the following example.
\begin{example}\label{exmaple order imply st conv}
	Let $E$ be a Riesz space and $(p_\alpha)_{\alpha\in A}$ be a net in $E$. Take $\mathcal{M}$ and $\mu$ from Example \ref{net give measure}. Thus, if $(p_\alpha)_{\alpha\in A}\downarrow x$ then $(p_\alpha)_{\alpha\in A}\stmu x$ for some $x\in E$.
\end{example}

For the general case of Example \ref{exmaple order imply st conv}, we give the following work whose proof follows directly from the basic definitions and results.
\begin{proposition}\label{order conv is st down to zero}
	If $(p_\alpha)_{\alpha\in A}$ is an order decreasing null net in a Riesz space then $(p_\alpha)_{\alpha\in A}\stmu 0$.
\end{proposition}

Now, we introduce the crucial notion of this paper.
\begin{definition}\label{main def}
	A net $(x_\alpha)_{\alpha\in A}$ in a Riesz space $E$ is said to be {\em $\mu$-statistically order convergent} to $x\in E$ if there exists a net $(p_\alpha)_{\alpha\in A}\in E_{st_\mu\downarrow}\{0\}$ with  $\mu(\Delta)=1$ for some $\Delta\in \mathcal{M}$ such that $(p_{\alpha_\delta})_{\delta\in \Delta}\downarrow 0$ and $|x_{\alpha_\delta}-x|\leq p_{\alpha_\delta}$ for every $\delta\in \Delta$. Then it is abbreviated as $x_\alpha\stmuc x$.
\end{definition}

It can be seen that the means of $x_\alpha\stmuc x$ in a Riesz space is equal to say that there exists another sequence $(p_\alpha)_{\alpha\in A}\stmu 0$ such that $\mu\big(\{\alpha\in A:|x_{\alpha}-x|\nleq p_{\alpha}\}\big)=0$. It follows from Remark \ref{natural density properties} that the notion of statistical convergence of nets coincides with the notion of $\mu$-statistically order convergence in reel line. We denote the set $E_{st_\mu}$ as the family of all $st_\mu$-convergent nets in $E$, and $E_{st_\mu}\{0\}$ is the family of all $\mu$-statistically order null nets in $E$.

\begin{lemma}\label{dec and conv}
	Every $\mu$-statistically decreasing net $\mu$-statistically order converges to its $\mu$-statistically decreasing limit in Riesz spaces.
\end{lemma}

\begin{remark}\label{relatively uni}
	Recall that a net $(x_{\alpha})_{\alpha \in A}$ in a Riesz space $E$ {\em relatively uniform converges} to $x\in E$ if there exists $u\in E_+$ such that, for any $n\in\mathbb{N}$, there is an index $\alpha_n\in A$ so that $|x_\alpha-x|\le\frac{1}{n} u$ for all $\alpha\ge\alpha_n$ (cf. \cite[Thm.16.2]{LZ}). It is well known that the relatively uniform convergence implies the order convergence on Archimedean Riesz spaces (cf. \cite[Lem.2.2]{AEG}). Hence, it follows from Proposition \ref{order conv is st down to zero} and Lemma \ref{dec and conv} that every decreasing relatively uniform null net is $\mu$-statistically order convergent in Riesz spaces.
\end{remark}

\section{Main Results}\label{sec3}

Let $\mu$ be a measure on a directed set $A$. Following from \cite[Exer.9. p.27]{Folland}, it is clear that $\mu(\Delta\cap \Sigma)=1$ for any $\Delta,\Sigma\subseteq A$, $\mu(\Delta)=\mu(\Sigma)=1$. So, we begin the section with the following observation.
\begin{proposition}\label{inequality}
	Assume $x_\alpha\leq y_\alpha\leq z_\alpha$ satisfies in a Riesz spaces for each index $\alpha$. Then $y_\alpha\stmuc x$ whenever $x_\alpha\stmuc x$ and $z_\alpha\stmuc x$.
\end{proposition}

It can be seen from Proposition \ref{inequality} that if $0\leq x_\alpha\leq z_\alpha$ satisfies for each index $\alpha$ and $(z_\alpha)_{\alpha\in A}\in E_{st_\mu}\{0\}$ then $(x_\alpha)_{\alpha\in A}\in E_{st_\mu}\{0\}$. We give a relation between the order and the $\mu$-statistically order convergences in the next result.
\begin{theorem}\label{order implies st conv}
	Every order convergent net is $\mu$-statistically order convergent to its order limit.
\end{theorem}

\begin{proof}
	Suppose that a net $(x_\alpha)_{\alpha\in A}$ is order convergent to $x$ in a Riesz space $E$. Then there exists another net $(y_\alpha)_{\alpha\in A}\downarrow 0$ such that $|x_\alpha-x|\le y_\alpha$ holds for all $\alpha\in A$. It follows from Proposition \ref{order conv is st down to zero} that $(y_\alpha)_{\alpha\in A}\stmu 0$. So, we obtain the desired result, $(x_\alpha)_{\alpha\in A}\stmuc x$.
\end{proof}

The converse of Theorem \ref{order implies st conv} need not to be true as the following example shows
\begin{example}\label{converse exam}
	Take the sequence $(e_n)$ of the standard unit vectors in the Riesz space $c_0$ the set of all null sequence in the real numbers. Then $(x_n)=(e_1,0,e_2,0,e_3,\cdots)$ does not converge in order because it is not order bounded in $c_0$. However, it can be seen that all eventually zero subsequences of $(e_n)$ is order convergent to zero, and so, they are $\mu$-statistically order convergent to zero. Then it follows Proposition \ref{subnet} that $x_n\stmuc 0$.
\end{example}

Following from Theorem \ref{order implies st conv} and \cite[Thm.23.2]{LZ}, we observe the following result.
\begin{corollary}\label{Dedekind riesz}
	Every order bounded monotone net in a Dedekind complete Riesz space is $\mu$-statistically order convergent.
\end{corollary}

\begin{proposition}\label{subnet}
	The $st_\mu$-convergence of subnets implies the $st_\mu$-convergence of nets.
\end{proposition}

\begin{proof}
	Let $(x_\alpha)_{\alpha\in A}$ be a net in a Riesz space $E$. Assume that a subnet $(x_{\alpha_\delta})_{\delta\in\Delta}$ of $(x_\alpha)_{\alpha\in A}$ $\mu$-statistically order converges to $x\in E$. Then there exists a net $(p_{\alpha_\delta})_{\delta\in\Delta}\in E_{st_\mu\downarrow}\{0\}$ with $\Sigma\subseteq\Delta$ and $\mu(\Sigma)=1$ such that $(p_{\alpha_{\delta_\sigma}})_{\sigma\in\Sigma}\downarrow 0$ and $|x_{\alpha_{\delta_\sigma}}-x|\leq p_{\alpha_{\delta_\sigma}}$ for each $\sigma\in\Sigma$. Since $\Sigma\subseteq A$ and $(x_{\alpha_{\delta_\sigma}})_{\sigma\in\Sigma}$ is also a subnet of $(x_\alpha)_{\alpha\in A}$, we can obtain the desired result.
\end{proof}

Since every order bounded net has an order convergent subnet in atomic $KB$-spaces, we give the following result by considering Theorem \ref{order implies st conv} and proposition \ref{subnet}.
\begin{corollary}\label{atomic}
	If $E$ is an atomic $KB$-space then every order bounded net is $\mu$-statistically order convergent in $E$.
\end{corollary}

The lattice operations are $\mu$-statistically order continuous in the following sense.
\begin{theorem}\label{st are continuous}
	If $x_\alpha\stmuc x$ and $w_\alpha\stmuc w$ then $x_\alpha\vee w_\alpha\stmuc x\vee w$.
\end{theorem}

\begin{proof}
	Assume that $x_\alpha\stmuc x$ and $w_\alpha\stmuc w$ hold in a Riesz space $E$. So, there are nets $(p_\alpha)_{\alpha\in A}, (q_\alpha)_{\alpha\in A}\in E_{st_\mu\downarrow}\{0\}$ with $\Delta,\Sigma\in\mathcal{M}$ and $\mu(\Delta)=\mu(\Sigma)=1$ such that  
	\begin{eqnarray*}
		|x_{\alpha_\delta}-x|\leq p_{\alpha_\delta} \  \text{and} \ |w_{\alpha_\sigma}-w|\leq q_{\alpha_\sigma}
	\end{eqnarray*}
	satisfy for all $\delta\in \Delta$ and $\sigma\in\Sigma$. On the other hand, it follows from \cite[Thm.1.9(2)]{ABPO} that the inequality $\lvert x_\alpha\vee w_\alpha-x\vee w\rvert\leq \lvert x_\alpha-x\rvert+\lvert w_\alpha-w\rvert$ holds for all $\alpha\in A$. Therefore, we have 
	$$
	\lvert x_{\alpha_\delta}\vee w_{\alpha_\sigma}-x\vee w\rvert\leq p_{\alpha_\delta}+q_{\alpha_\sigma}
	$$
	for each $\delta\in \Delta$ and $\sigma\in\Sigma$. Take $\Gamma:=\Delta\cap \Sigma\in\mathcal{M}$. So, we have $\mu(\Gamma)=1$, and also, $\lvert x_{\alpha_\gamma}\vee w_{\delta_\gamma}-x\vee w\rvert\leq p_{\alpha_\gamma}+q_{\delta_\gamma}$ holds for all $\gamma\in \Gamma$. It follows from $(p_{\alpha_\gamma}+q_{\delta_\gamma})_{\gamma\in\Gamma}\downarrow 0$ that $x_\alpha\vee w_\alpha\stmuc x\vee w$.
\end{proof}

\begin{corollary}\label{basic corollary}
	If $x_\alpha\stmuc x$ and $w_\alpha\stmuc w$ in a Riesz space then
	\begin{enumerate}
		\item[(i)] \ $x_\alpha\wedge w_\alpha\stmuc x\wedge w$;
		\item[(ii)] \ $|x_\alpha|\stmuc |x|$;
		\item[(iii)] \ $x_\alpha^+\stmuc x^+$;
		\item[(iv)] \ $x_\alpha^-\stmuc x^-$.
	\end{enumerate}
\end{corollary}

We continue with several basic results that are motivated by their analogies from Riesz space theory.
\begin{theorem}\label{basic properties}
	Let $(x_\alpha)_{\alpha\in A}$ be a net in a Riesz space $E$. Then the following results hold:
	\begin{enumerate}
		\item[(i)] \ $x_\alpha\stmuc x$ iff  $(x_\alpha-x)\stmuc 0$ iff $\lvert x_\alpha-x\rvert\stmuc 0$;
		\item[(ii)] \ the $\mu$-statistically order limit is linear;
		\item[(iii)] \ the $\mu$-statistically order limit is uniquely determined;
		\item[(iv)] \ the positive cone $E_+$ is closed under the $\mu$-statistically order convergence;
		\item[(v)] \ $x_{\alpha_\delta}\stmuc x$ for any subnet $(x_{\alpha_\delta})_{\delta\in \Delta}$ of $x_\alpha\stmuc x$ with $\mu(\Delta)=1$.
	\end{enumerate}
\end{theorem}

\begin{proof}
	The properties $(i),\ (ii)$ and $(iii)$ are straightforward.
	
	For $(iv)$, take a non-negative $\mu$-statistically order convergent net $x_\alpha\stmuc x$ in $E$. Then it follows from Corollary \ref{basic corollary} that $x_\alpha=x_\alpha^+\stmuc x^+$. Moreover, by applying $(iii)$, we have $x=x^+$. So, we obtain the desired result $x\in E_+$.
	
	For $(v)$, suppose that $x_\alpha\stmuc x$. Then there is a net $(p_\alpha)_{\alpha\in A}\in E_{st_\mu\downarrow}\{0\}$ with $\Delta\in\mathcal{M}$ and $\mu(\Delta)=1$ such that  $|x_{\alpha_\delta}-x|\leq p_{\alpha_\delta}$ for each $\delta\in \Delta$. Thus, it is clear that $x_{\alpha_\delta}\stmuc x$. However, it should be shown that it is provided for all subnets under our assumption. Thus, take an arbitrary element $\Sigma\in \mathcal{M}$ with $\Sigma\neq\Delta$ and  $\mu(\Sigma)=1$. We show $(x_{\alpha_\sigma})_{\sigma\in\Sigma}\stmuc x$. Consider $\Gamma:=\Delta\cap \Sigma\in \mathcal{M}$. So, we have $\mu(\Gamma)=1$. Therefore, following from $|x_{\alpha_\gamma}-x|\leq p_{\alpha_\gamma}$ for each $\gamma\in \Gamma$, we get the desired result.
\end{proof}

Theorem \ref{order implies st conv} shows that the order convergence implies the $\mu$-statistically order convergence. For the converse of this, we give the following result.
\begin{theorem}\label{monoton isord conv}
	Every monotone $\mu$-statistically order convergent net order converges to its ${st}_\mu$-limit in Riesz spaces.
\end{theorem}

\begin{proof}
	We show that $x_\alpha\downarrow$ and $x_\alpha\stmuc x$ implies $x_\alpha\downarrow x$ in any Riesz space $E$. To see this, choose an arbitrary index $\alpha_0$. Then $x_{\alpha_0}-x_\alpha\in E_+$ for all $\alpha\geq \alpha_0$. It follows from Theorem \ref{basic properties} that $x_{\alpha_0}-x_\alpha\stmuc x_{\alpha_0}-x$, and also, $x_{\alpha_0}-x\in E_+$. Hence, we have $x_{\alpha_0}\geq x$. Then $x$ is a lower bound of $(x_\alpha)_{(\alpha\in A)}$ because $\alpha_0$ is arbitrary. Suppose that $z$ is another lower bound of $(x_\alpha)_{\alpha\in A}$. So, we obtain $x_\alpha-z\stmuc x-z$. It means that $x-z\in E_+$, or equivalent to saying that $x\geq z$. Therefore, we get $x_\alpha \downarrow x$.
\end{proof}

\begin{theorem}\label{cracter}
	Let $x:=(x_\alpha)_{\alpha\in A}$ be a net in a Riesz space. If $x\mathcal{X}_\Delta\co 0$ holds for some $\Delta\in \mathcal{M}$ with $\mu(\Delta)=1$ then $x\stmuc 0$.
\end{theorem}

\begin{proof}
	Suppose that there exists $\Delta\in\mathcal{M}$ with $\mu(\Delta)=1$ and $x\mathcal{X}_\Delta\co 0$ satisfies in a Riesz space $E$ for the characteristic function $\mathcal{X}_\Delta$ of $\Delta$. Thus, there is another net $(p_\alpha)_{\alpha\in A}\downarrow 0$ such that $|x\mathcal{X}_\Delta|\le p_\alpha$ for all $\alpha\in A$. So, it follows from Proposition \ref{order conv is st down to zero} that $(p_\alpha)_{\alpha\in A}\stmu 0$. Then there exists a subset $\Sigma\in\mathcal{M}$ such that $\mu(\Sigma)=1$ and $(p_{\alpha_\sigma})_{\sigma\in \Sigma}\downarrow 0$. Take $\Gamma:=\Delta\cap \Sigma$. Hence, we have $\mu(\Gamma)=1$. Following from $|x\mathcal{X}_\Gamma|\leq p_{\alpha_\gamma}$ for each $\gamma\in \Gamma$, we obtain $x\mathcal{X}_\Delta\stmuc 0$. Therefore, by applying Proposition \ref{subnet}, we obtain $(x_\alpha)_{\alpha\in A}\stmuc 0$.
\end{proof}

\begin{proposition}\label{riesz space}
	The family of all $st_\mu$-convergent nets $E_{st_\mu}$ is a Riesz space.
\end{proposition}

\begin{proof}
	From the properties in Theorem \ref{basic properties}, $E_{st_\mu}$ is a vector space. Take an element $x:=(x_\alpha)_{\alpha_\in A}$ in $E_{st_\mu}$. Then we have $x\stmuc z$ for some $z\in E$. Thus, it follows from Corollary \ref{basic corollary} that $|x|\stmuc |z|$. It means that $|x|\in E_{st_\mu}$, i.e., $E_{st_\mu}$ is a Riesz subspace (cf. \cite[Thm.1.3 and Thm.1.7]{AB}).
\end{proof}

\begin{theorem}\label{ideal}
	The set of all order bounded nets in a Riesz space $E$ is an order ideal in $E_{st_\mu}\{0\}$.
\end{theorem}

\begin{proof}
	By the same argument in Proposition \ref{riesz space}, $E_{st_\mu}\{0\}$ is a Riesz space. Assume that $|y|\leq |x|$ hold for arbitrary $x:=(x_\alpha)_{\alpha_\in A}\in E_{st_\mu}\{0\}$ and for an order bounded net $y:=(y_\alpha)_{\alpha_\in A}$. Since $x\stmuc 0$, we have $|x|\stmuc 0$. Then it follows from Proposition \ref{inequality} that $|y|\stmuc 0$, and so, it follows from Theorem \ref{basic properties}$(i)$ that $y\stmuc 0$. Therefore, we get the desired result, $y\in E_{st_\mu}\{0\}$.
\end{proof}



\begin{thebibliography}{30}

\bibitem{AB}
C. D. Aliprantis, O. Burkinshaw, Locally Solid Riesz Spaces with Applications to Economics, Amer. Math. Soc., Providence, RI, 2003.

\bibitem{ABPO}
C. D. Aliprantis, O. Burkinshaw, Positive Operators, Springer, Dordrecht, 2006.

\bibitem{AS}
Y. A. Abramovich, G. Sirotkin, On order convergence of nets, Positivity, 9 (3) (2005), 287-292.

\bibitem{Aydn1}
A. Ayd\i n, The statistically unbounded $\tau$-convergence on locally solid Riesz spaces, Turk. J. Math. \textbf{44}(3), 949-956, 2020.

\bibitem{Aydn2}
A. Ayd\i n, The statistical multiplicative order convergence in Riesz algebras, to appear Fact. Univ. Ser.: Math. Infor. 2021.

\bibitem{AEG}
A. Ayd{\i}n, E. Emelyanov, S. G. Gorokhova, Full lattice convergence on Riesz spaces, Indagat. Math. \textbf{32} (3) (2021), 658-690.

\bibitem{AC}
A. Ayd\i n, M. Et, Statistically multiplicative convergence on locally solid Riesz algebras, to appear Turk. J. Math. 2021.

\bibitem{CLL}
L. Cheng, G. Lin, Y. Lan, H. Lui, Measure theory of statistical convergence, Sci. China Ser. A-Math. \textbf{51}(12), 2285-2303, 2008.

\bibitem{Con}
J. Connor, Two valued measures and summability, Analysis, \textbf{10}(4), 373-385, 1990.

\bibitem{Con2}
J. Connor, $R$-type summability methods, Cauchy criteria, $P$-sets and statistical convergence, Proc Amer. Math Soc. \textbf{115}(2), 319–327 (1992)

\bibitem{DO}
O. Duman, C. Orhan, $\mu$-Statistically Convergent Function Sequences, Czech. Math. \textbf{54}(2), 413–422, 2004.

\bibitem {Fast}
H. Fast, Sur la convergence statistique, Colloq. Math. \textbf{2}, 241-244, 1951.

\bibitem {FS}
A. R. Freedman, J. J. Sember, Densities and summability, Pacif. J. Math. \textbf{95}(2), 293-305, 1981.

\bibitem {Fridy}
J. A. Fridy, On statistical convergence, Analysis, \textbf{5}(4), 301-313, 1985.

\bibitem {Folland}
G. B. Folland, Real Analysis: Modern Techniques and Their Applications, Wiley, 2013.

\bibitem{LZ}
W. A. J. Luxemburg, A. C. Zaanen, Riesz Spaces I, North-Holland Pub. Co., Amsterdam, 1971.

\bibitem{M}
I. J. Maddox, Statistical convergence in a locally convex space,
Math. Proc. Cambr. Phil. Soc. \textbf{104}(1), 141-145, 1988.

\bibitem{Mor}
F. Moricz, Statistical limits of measurable functions, Analysis, \textbf{24}(1), 1-18, 2004.

\bibitem{MDM}
A. R. Murugan, J. Dianavinnarasi, C. Ganesa Moorthy, Statistical Convergence of Nets Through Directed Sets, Univ. J. Math. App. \textbf{2}(2), 79-84, 2019.

\bibitem{Rud}
V. Runde, A Taste of Topology, Springer, New York, 2005.

\bibitem{SAL}
T. Salat, On statistically convergent sequences of real numbers, Math. Slov. \textbf{30}(2), 139-150, 1980.

\bibitem{St}
H. Steinhaus, Sur la convergence ordinaire et la convergence asymptotique, Colloq. Math. \textbf{2}, 73-74, 1951.

\bibitem{SP}
C. Şençimen, S. Pehlivan, Statistical order convergence in Riesz spaces,
Math. Slov. \textbf{62}(2), 557-570, 2012.

\bibitem{Za}
A. C. Zaanen, Riesz Spaces II, North-Holland Pub. Co. Amsterdam, 1983.
\end{thebibliography}
\end{document}